\documentclass[oneside,english]{amsart}
\usepackage[T1]{fontenc}
\usepackage[latin9]{inputenc}
\usepackage{amsthm}
\usepackage{amstext}
\usepackage{amssymb}
\usepackage{esint}
\usepackage{graphicx}
\usepackage{enumerate}
\usepackage{amscd}

\makeatletter
\newtheorem{theorem}{Theorem}[section]
\newtheorem{lemma}[theorem]{Lemma}

\numberwithin{figure}{section}
\theoremstyle{definition}

\theoremstyle{remark}
\newtheorem{remark}[theorem]{Remark}

\numberwithin{equation}{section}
\makeatother

\usepackage{color}

\usepackage{babel}

	\DeclareMathOperator{\dist}{dist}
	\DeclareMathOperator{\loc}{loc}
	
	\DeclareMathOperator*{\esssup}{ess\,sup}

\begin{document}

\title[Estimates for the principal Dirichlet-Laplacian eigenvalue]{On growth monotonicity estimates of the principal Dirichlet-Laplacian eigenvalue}

\author{V.~A.~Pchelintsev}

\begin{abstract}
In the present paper we obtain growth monotonicity estimates of the principal Dirichlet-Laplacian eigenvalue in bounded non-Lipschitz domains. The proposed method is based on composition operators generated by quasiconformal mappings and their applications to weighted Sobolev inequalities.
\end{abstract}
\maketitle
\footnotetext{\textbf{Key words and phrases:} Elliptic equations, Sobolev spaces, quasiconformal mappings.}
\footnotetext{\textbf{2010
Mathematics Subject Classification:} 35P15, 46E35, 30C65.}

\section{Introduction}

The goal of present work is to obtain the growth monotonicity estimates of the principal eigenvalue of the Dirichlet-Laplacian
\[
- \textrm{div}(\nabla u)=\lambda u\,\, \text{in} \,\, \Omega, \quad u=0 \,\, \text{on} \,\, \partial \Omega,
\]
generated by quasiconformal deformations of bounded non-Lipschitz domains $\Omega\subset\mathbb R^2$ satisfying to the quasihyperbolic boundary condition \cite{KOT01,KOT02}. This class of domains includes domains with non-rectifiable boundaries \cite{R01}.

Recall that a domain $\Omega$ satisfies the $\gamma$-quasihyperbolic boundary condition with some $\gamma>0$, if the growth condition on the quasihyperbolic metric
$$
k_{\Omega}(x_0,x)\leq \frac{1}{\gamma}\log\frac{\dist(x_0,\partial\Omega)}{\dist(x,\partial\Omega)}+C_0
$$
is satisfied for all $x\in\Omega$, where $x_0\in\Omega$ is a fixed base point and $C_0=C_0(x_0)<\infty$,
\cite{GM,H1}.

We will consider the Dirichlet eigenvalue problem in the weak formulation:
\[
\int\limits_{\Omega} \left\langle \nabla u(x), \nabla v(x) \right\rangle dx
= \lambda \int\limits_{\Omega} u(x) v(x)~dx, \quad v \in W_0^{1,2}(\Omega).
\]

It is known \cite{Henr,M} that in bounded domains $\Omega \subset \mathbb R^2$ the Dirichlet spectrum of the Laplace operator is discrete and can be written in the form of a non-decreasing sequence
\[
0< \lambda_1(\Omega) \leq \lambda_2(\Omega) \leq \ldots \leq \lambda_n(\Omega) \leq \ldots ,
\]
where each eigenvalue is repeated as many time as its multiplicity. So, by the Min-Max Principle (see, for example, \cite{Henr,M}) the first eigenvalue of the Dirichlet-Laplacian can be represented as
\[
\lambda_1(\Omega)= \inf_{u \in W_0^{1,2}(\Omega) \setminus \{0\}} \frac{\int\limits_{\Omega} |\nabla u(x)|^2~dx}{\int\limits_{\Omega} |u(x)|^2~dx}\,.
\]

It is known that exact calculations of Dirichlet-Laplacian eigenvalues are possible in a limited number of cases. So, estimates of Dirichlet eigenvalues are significant in the spectral theory of elliptic operators.

The lower estimates of Dirichlet eigenvalues directly connected to the Rayleigh-Faber-Krahn inequality \cite{F23,Kh25} which states that the disc minimizes the principal Dirichlet-Laplacian eigenvalue among all planar domains of the same area:
\begin{equation*}
\lambda_1(\Omega)\geq \lambda_1(\Omega^{\ast})=\frac{{j_{0,1}^2}}{R^2_{\ast}},
\end{equation*}
where $j_{0,1} \approx 2.4048$ is the first positive zero of the Bessel function $J_0$ and $\Omega^{\ast}$ is a disc of the same area as $\Omega$ with $R_{\ast}$ as its radius. Note that the Rayleigh-Faber-Krahn inequality was refined using the capacity method, see $\S 4.7$ in \cite{M}.

Another lower bound for the principal Dirichlet-Laplacian eigenvalue for a simply connected planar domain was obtained by Makai \cite{M65}:
\begin{equation}\label{M-H}
\lambda_1(\Omega)\geq \frac{\alpha}{\rho^2},
\end{equation}
where $\alpha =1/4$ and $\rho$ is the radius of the largest disc inscribed in $\Omega$. For convex domains, this lower bound with $\alpha = \pi^2/4$ was obtained by Hersch \cite{H60}.

In this paper we proposed estimates for the principal Dirichlet-Laplacian eigenvalue in terms of the quasihyperbolic geometry. For this we will use the method is based on connections between the quasiconformal mappings \cite{Ahl66} and composition operators on Sobolev spaces \cite{GG94,GU09,U93,VU02}. It permits us to obtain growth monotonicity estimates of the principal Dirichlet eigenvalue in domains with quasihyperbolic boundary conditions.

\vskip 0.2cm

The main results of the article:

\vskip 0.2cm

\textit{
If a non-Lipschitz domain $\Omega \subset \mathbb R^2$ satisfies to the quasihyperbolic boundary condition, then
\begin{equation}\label{low-est}
\lambda_1(\Omega) \geq \frac{\lambda_1(\Omega')}{K ||J_{\psi}\,|\,L^{\infty}(\Omega')||},
\end{equation}
where $\Omega' \subset \mathbb R^2$ is a bounded domain and $J_{\psi}$ is a Jacobian of the $K$-quasiconformal mapping $\psi:\Omega' \to \Omega$.}

\vskip 0.2cm

Let us illustrate the above result with an example. The homeomorphism
\[
\varphi(z)= \sqrt{a^2+1}z+a \overline{z}, \quad z=x+iy, \quad a\geq 0,
\]
is a $K$-quasiconformal with $K=\frac{\sqrt{a^2+1}+a}{\sqrt{a^2+1}-a}$ and maps the unit disc $\mathbb D$ onto the interior of ellipse $\Omega_e$ with semi-axes $\sqrt{a^2+1}+a$ and $\sqrt{a^2+1}-a$.
In this case the Jacobian $J(z,\varphi)=|\varphi_{z}|^2-|\varphi_{\overline{z}}|^2=1$. Hence, we have
\[
\lambda_1(\Omega_e) \geq \frac{\sqrt{a^2+1}-a}{\sqrt{a^2+1}+a} j_{0,1}^2,
\]
where $j_{0,1} \approx 2.4048$ is the first positive zero of the Bessel function $J_0$. Note that for $a \leq 1/8$ this estimate is better than estimate \eqref{M-H} with $\alpha = \pi^2/4$, i.e.
\[
\frac{\sqrt{a^2+1}-a}{\sqrt{a^2+1}+a} j_{0,1}^2 > \frac{\pi^2}{4(\sqrt{a^2+1}-a)^2}.
\]

\vskip 0.2cm

A classical result (see, for example, \cite{Henr,M}) tells us that the Dirichlet eigenvalue $\lambda(\Omega)$ satisfies a so-called domain monotonicity property, namely
\[
\text{if} \quad \Omega_1 \subseteq \Omega_2 \quad \text{then} \quad \lambda(\Omega_1)- \lambda(\Omega_2) \geq 0.
\]

Given this property and inequality \eqref{low-est} we obtain growth monotonicity estimates of the principal Dirichlet eigenvalue of the Dirichlet-Laplacian.

\vskip 0.2cm

\textit{
If a non-Lipschitz domain $\Omega \subset \mathbb R^2$ satisfies to the quasihyperbolic boundary condition and such that $\Omega \subset \Omega'$, then
\[
\lambda_1(\Omega)-\lambda_1(\Omega') \geq \left(\frac{1}{K ||J_{\psi}\,|\,L^{\infty}(\Omega')||}-1\right)\lambda_1(\Omega'),
\]
where $\Omega' \subset \mathbb R^2$ is a bounded domain and $J_{\psi}$ is a Jacobian of the $K$-quasiconformal mapping $\psi:\Omega' \to \Omega$.}

\vskip 0.2cm

As an example, we consider growth monotonicity estimates of the principal Dirichlet-Laplacian eigenvalue in domains type rose petals $\Omega_a$, where
$$
\Omega_a:=\left\{(\rho, \theta) \in \mathbb R^2:\rho=2a\cos(2 \theta), \quad -\frac{\pi}{4} \leq \theta \leq \frac{\pi}{4}\right\}
$$
is the image of the unit disc $\mathbb D$ under the $2$-quasiconformal mapping
\[
\psi(z)=a (z+1)^{\frac{3}{4}} (\overline{z}+1)^{\frac{1}{4}}, \quad z=x+iy, \quad 0<a<1.
\]
In this case the quality $K ||J_{\psi}\,|\,L^{\infty}(\mathbb D)||=a^2<1$ and $\mathbb D \supset \Omega_a$. Hence we have
\[
\lambda_1(\Omega_a)-\lambda_1(\mathbb D) \geq \frac{1-a^2}{a^2} j_{0,1}^2,
\]
where $j_{0,1} \approx 2.4048$ is the first positive zero of the Bessel function $J_0$.

Others examples growth monotonicity estimates of the principal Dirichlet-Laplacian eigenvalue in non-Lipschitz domains will be given in Section 4.

In works \cite{BGU15,BGU16,GPU17,GPU19,GU16,GU17} using approaches are based on the geometric theory of composition operators on Sobolev spaces were obtained the spectral estimates for Dirichlet and Neumann eigenvalues of the Laplace operator for a large class of rough domains satisfying quasihyperbolic boundary conditions. These composition operators are generated by conformal mappings, quasiconformal mappings and their generalizations.

\section{Sobolev spaces and quasiconformal mappings}

In this section we recall basic facts about composition operators on Lebesgue and Sobolev spaces and also the quasiconformal mappings theory.

The following theorem about composition operators on Lebesgue spaces is well known (see, for example, \cite{VU02}):
\begin{theorem}
Let $\varphi :\Omega \to {\Omega'}$ be a weakly differentiable homeomorphism between two domains $\Omega$ and $\Omega'$.
Then the composition operator
\[
\varphi^{*}: L^r(\Omega') \to L^s(\Omega),\,\,\,1 \leq s \leq r< \infty,
\]
is bounded, if and only if $\varphi^{-1}$ possesses the Luzin $N$-property and
\[
\biggr(\int\limits _{\Omega'}\left|J(y,\varphi^{-1})\right|^{\frac{r}{r-s}}\, dy\biggr)^{\frac{r-s}{rs}}=K< \infty,\,\,\,1 \leq s<r< \infty,
\]
\[
\esssup\limits_{y \in \Omega'}\left|J(y,\varphi^{-1})\right|^{\frac{1}{s}}=K< \infty,\,\,\,1 \leq s=r< \infty.
\]
The norm of the composition operator $\|\varphi^{*}\|=K$.
\end{theorem}

Let $\Omega\subset\mathbb R^n$ be an open set. The Sobolev space $W^{1,p}(\Omega)$, $1\leq p\leq\infty$ is defined
as a Banach space of locally integrable weakly differentiable functions
$f:\Omega\to\mathbb{R}$ equipped with the following norm:
\[
\|f\mid W^{1,p}(\Omega)\|=\biggr(\int\limits _{\Omega}|f(x)|^{p}\, dx\biggr)^{\frac{1}{p}}+
\biggr(\int\limits _{\Omega}|\nabla f(x)|^{p}\, dx\biggr)^{\frac{1}{p}},
\]
where $\nabla f$ is the weak gradient of the function $f$. Recall that the Sobolev space $W^{1,p}(\Omega)$ coincides with the closure of the space of smooth functions $C^{\infty}(\Omega)$ in the norm of $W^{1,p}(\Omega)$.

The homogeneous seminormed Sobolev space $L^{1,p}(\Omega)$, $1\leq p<\infty$,
is the space of all locally integrable weakly differentiable functions equipped
with the following seminorm:
\[
\|f\mid L^{1,p}(\Omega)\|=\biggr(\int\limits _{\Omega}|\nabla f(x)|^{p}\, dx\biggr)^{\frac{1}{p}}.
\]

The Sobolev space $W^{1,p}_{0}(\Omega)$, $1 \leq p< \infty$, is the closure in the $W^{1,p}(\Omega)$-norm of the
space $C^{\infty}_{0}(\Omega)$ of all infinitely continuously differentiable functions with compact support in $\Omega$.

We consider the Sobolev spaces as Banach spaces of equivalence classes of functions up to a set of $p$-capacity zero \cite{M}.

Let $\Omega$ and ${\Omega'}$ be domains in $\mathbb R^n$. We say that
a homeomorphism $\varphi:\Omega\to{\Omega'}$ induces by the composition rule $\varphi^{\ast}(f)=f\circ\varphi$ a bounded composition operator
\[
\varphi^{\ast}:L^1_p({\Omega'})\to L^1_q(\Omega),\,\,\,1\leq q\leq p\leq\infty,
\]
if the composition $\varphi^{\ast}(f)\in L^1_q(\Omega)$
is defined quasi-everywhere in $\Omega$ and there exists a constant $K_{p,q}(\Omega)<\infty$ such that
\[
\|\varphi^{\ast}(f)\mid L^1_q(\Omega)\|\leq K_{p,q}(\Omega)\|f\mid L^1_p({\Omega'})\|
\]
for
any function $f\in L^1_p({\Omega'})$ \cite{U93}.

Recall that a mapping $\varphi:\Omega\to\mathbb R^n$ belongs to $L^{1,p}_{\loc}(\Omega)$,
$1\leq p\leq\infty$, if its coordinate functions $\varphi_j$ belong to $L^{1,p}_{\loc}(\Omega)$, $j=1,\dots,n$.
In this case the formal Jacobi matrix
$D\varphi(x)=\left(\frac{\partial \varphi_i}{\partial x_j}(x)\right)$, $i,j=1,\dots,n$,
and its determinant (Jacobian) $J(x,\varphi)=\det D\varphi(x)$ are well defined at
almost all points $x\in \Omega$. The norm $|D\varphi(x)|$ of the matrix
$D\varphi(x)$ is the norm of the corresponding linear operator $D\varphi (x):\mathbb R^n \rightarrow \mathbb R^n$ defined by the matrix $D\varphi(x)$.

Let $\varphi:\Omega\to{\Omega'}$ be weakly differentiable in $\Omega$. The mapping $\varphi$ is the mapping of finite distortion if $|D\varphi(z)|=0$ for almost all $x\in Z=\{z\in\Omega : J(z,\varphi)=0\}$.

A mapping $\varphi:\Omega\to\mathbb R^n$ possesses the Luzin $N$-property if a image of any set of measure zero has measure zero.
Mote that any Lipschitz mapping possesses the Luzin $N$-property.

The following theorem gives the analytic description of composition operators on Sobolev spaces:

\begin{theorem}
\label{CompTh} \cite{U93} A homeomorphism $\varphi:\Omega\to\Omega'$
between two domains $\Omega$ and $\Omega'$ induces a bounded composition
operator
\[
\varphi^{\ast}:L^{1,p}(\Omega')\to L^{1,q}(\Omega),\,\,\,1\leq q< p<\infty,
\]
 if and only if $\varphi\in W_{\loc}^{1,1}(\Omega)$, has finite distortion,
and
$$
K_{p,q}(\Omega)=\left(\int\limits_\Omega \left(\frac{|D\varphi(x)|^p}{|J(x,\varphi)|}\right)^\frac{q}{p-q}~dx\right)^\frac{p-q}{pq}<\infty.
$$
\end{theorem}
For non-homeomorphic mappings a similar result was obtain in \cite{VU02}.

Recall that a homeomorphism $\varphi: \Omega\to \Omega'$ is called a $K$-quasiconformal mapping if $\varphi\in W^{1,n}_{\loc}(\Omega)$ and there exists a constant $1\leq K<\infty$ such that
$$
|D\varphi(x)|^n\leq K |J(x,\varphi)|\,\,\text{for almost all}\,\,x\in\Omega.
$$

Note that quasiconformal mappings have a finite distortion and a mapping which is inverse to a quasiconformal mapping is also quasiconformal \cite{VGR}.

If $\varphi : \Omega \to \Omega'$ is a $K$-quasiconformal mapping then $\varphi$ is differentiable almost everywhere in $\Omega$ and
$$
|J(x,\varphi)|=J_{\varphi}(x):=\lim\limits_{r\to 0}\frac{|\varphi(B(x,r))|}{|B(x,r)|}\,\,\text{for almost all}\,\,x\in\Omega.
$$

\section{Weighted Sobolev inequality}

First of all, we recall two results concerning the non-weighted Sobolev inequality for a bounded domain ${\Omega'}\subset\mathbb R^2$ \cite{GPU2019} and a connection between composition operators on Sobolev spaces and the quasiconformal mappings theory \cite{VG75}. Namely:
\begin{theorem}
\label{PoinConst}
Let ${\Omega'}\subset\mathbb R^2$ be a bounded domain and $f \in W^{1,2}_0({\Omega'})$. Then
\begin{equation}\label{InPS}
\|f \mid L^{r}({\Omega'})\| \leq A_{r,2}({\Omega'}) \|\nabla f \mid L^{2}({\Omega'})\|, \,\,r \geq 2,
\end{equation}
where
\[
A_{r,2}({\Omega'}) \leq \inf\limits_{p\in \left(\frac{2r}{r+2},2\right)}
\left(\frac{p-1}{2-p}\right)^{\frac{p-1}{p}}
\frac{\left(\sqrt{\pi}\cdot\sqrt[p]{2}\right)^{-1}|{\Omega'}|^{\frac{1}{r}}}{\sqrt{\Gamma(2/p) \Gamma(3-2/p)}}.
\]
\end{theorem}

\begin{lemma} \label{L4.1}
A homeomorphism $\varphi: \Omega\to {\Omega'}$ is a $K$-quasiconformal mapping if and only if $\varphi$ generates, by the composition rule $\varphi^{\ast}(f)=f\circ\varphi$, an isomorphism of Sobolev spaces $L^{1,n}(\Omega)$ and $L^{1,n}({\Omega'})$:
$$
\|\varphi^{\ast}(f) \mid L^{1,n}(\Omega)\|\leq K^{\frac{1}{n}}\|f \mid L^{1,n}({\Omega'})\|
$$
for any $f\in L^{1,n}({\Omega'})$.
\end{lemma}

Taking into account Theorem~\ref{PoinConst} and Lemma~\ref{L4.1} we obtain an universal weighted Sobolev inequality which holds in any simply connected planar domain with non-empty boundary. Denote by $h(x) =|J(x,\varphi)|$  the quasihyperbolic weight defined by a $K$-quasiconformal mapping $\varphi : \Omega \to \Omega'$.

\begin{theorem}\label{Th4.1}
Let $\Omega$ be a simply connected planar domain.
Then for any function $f \in W^{1,2}_{0}(\Omega)$, the weighted Sobolev-Poincar\'e inequality
\[
\left(\int\limits_\Omega |f(x)|^rh(x)dx\right)^{\frac{1}{r}} \leq A_{r,2}(h,\Omega)
\left(\int\limits_\Omega |\nabla f(x)|^2 dx\right)^{\frac{1}{2}}
\]
holds for any $r \geq 2$ with the constant $A_{r,2}(h,\Omega) \leq K^{\frac{1}{2}} A_{r,2}({\Omega'})$.
\end{theorem}

\begin{proof}
By \cite{Ahl66} there exists a $K$-quasiconformal homeomorphism $\varphi : \Omega \to {\Omega'}$.

So, by Lemma~\ref{L4.1} the inequality
\begin{equation}\label{IN2.1}
||f \circ \varphi^{-1} \,|\, L^{1,2}({\Omega'})|| \leq K^{\frac{1}{2}} ||f \,|\, L^{1,2}(\Omega)||
\end{equation}
holds for any function $f \in L^{1,2}(\Omega)$.

Put $h(x):=|J(x,\varphi)|$.
Using the change of variable formula for the quasiconformal mappings \cite{VGR}, the inequality \eqref{IN2.1} and Theorem~\ref{PoinConst}, we get that for any smooth function $f\in L^{1,2}(\Omega)$
\begin{multline*}
\left(\int\limits_\Omega |f(x)|^rh(x)dx\right)^{\frac{1}{r}}
=\left(\int\limits_\Omega |f(x)|^r |J(x,\varphi)| dx\right)^{\frac{1}{r}} \\
=\left(\int\limits_{\Omega'} |f \circ \varphi^{-1}(y)|^rdy\right)^{\frac{1}{r}}
\leq A_{r,2}({\Omega'})
\left(\int\limits_{\Omega'} |\nabla (f \circ \varphi^{-1}(y))|^2dy\right)^{\frac{1}{2}} \\
\leq K^{\frac{1}{2}} A_{r,2}(\widetilde{\Omega})
\left(\int\limits_{\Omega} |\nabla f(x)|^2 dx\right)^{\frac{1}{2}}.
\end{multline*}

Approximating an arbitrary function $f \in W^{1,2}_{0}(\Omega)$ by smooth functions we have
$$
\left(\int\limits_\Omega |f(x)|^rh(x)dx\right)^{\frac{1}{r}} \leq
A_{r,2}(h,\Omega) \left(\int\limits_{\Omega} |\nabla f(x)|^2 dx\right)^{\frac{1}{2}},
$$
with the constant
$$
A_{r,2}(h,\Omega) \leq K^{\frac{1}{2}} A_{r,2}({\Omega}) \leq K^{\frac{1}{2}} \inf\limits_{p\in \left(\frac{2r}{r+2},2\right)}
\left(\frac{p-1}{2-p}\right)^{\frac{p-1}{p}}
\frac{\left(\sqrt{\pi}\cdot\sqrt[p]{2}\right)^{-1}|{\Omega'}|^{\frac{1}{r}}}{\sqrt{\Gamma(2/p) \Gamma(3-2/p)}}.
$$
\end{proof}

\subsection{ Estimates of Sobolev constants} In this section we consider (sharp) upper estimates of Sobolev constants in domains that satisfy the quasihyperbolic boundary condition.

In \cite{AK} it was proved that Jacobians of  quasiconformal mappings $\varphi: \Omega \to \Omega'$ belong to $L^{\beta}(\Omega)$ for some $\beta>1$ if and only if $\Omega'$ satisfy to a $\gamma$-quasihyperbolic boundary conditions for some $\gamma$. Note that the degree of integrability $\beta$ depends only on
$\Omega$ and the quasiconformality coefficient $K(\varphi)$.

Since we need the exact value of the integrability exponent $\beta$ for quasiconformal Jacobians, we consider an equivalent definition (of domains satisfying the quasihyperbolic boundary condition) in terms of integrability of Jacobians \cite{GPU19}.
A simply connected domain $\Omega'$ is called a $K$-quasi\-con\-for\-mal $\beta$-regular domain about a domain $\Omega$ if there exists a $K$-quasiconformal mapping $\varphi : \Omega \to \Omega'$ such that
{
$$
\int\limits_{\Omega} |J(x,\varphi)|^{\beta}~dx < \infty \quad\text{for some}\quad \beta >1,
$$
}
where $J(x,\varphi)$ is a Jacobian of a $K$-quasiconformal mapping $\varphi : \Omega \to \Omega'$.
The domain $\Omega' \subset \mathbb{R}^2$ is called a quasiconformal regular domain if it is a $K$-quasiconformal $\beta$-regular domain for some $\beta>1$.

The following theorem gives (sharp) upper estimates of (non-weighted) Sobolev constants in $K$-quasiconformal $\infty$-regular domains.
\begin{theorem}\label{Th4.3}
Let $\Omega$ be a $K$-quasi\-con\-formal $\infty$-regular domain about a domain ${\Omega'}$. Then for any function $f \in W^{1,2}_{0}(\Omega)$, the Sobolev inequality
\[
\|f\mid L^2(\Omega)\| \leq A_{2,2}(\Omega)
\|f\mid L^{1,2}(\Omega)\|
\]
holds with the constant
$$A_{2,2}(\Omega) \leq K^{\frac{1}{2}} A_{2,2}({\Omega'}) \big\|J_{\varphi^{-1}}\mid L^{\infty}({\Omega'})\big\|^{\frac{1}{2}},
$$
where $J_{\varphi^{-1}}$ is a Jacobian of the $K$-quasiconformal mapping $\varphi^{-1}:{\Omega'}\to\Omega$.
\end{theorem}

\begin{remark}
The constant $A_{2,2}^2({\Omega'})=1/\lambda_1({\Omega'})$, where $\lambda_1({\Omega'})$ is the first Dirichlet eigenvalue of Laplacian in a domain ${\Omega'}\subset\mathbb R^2$.
\end{remark}

\begin{proof}
Let a function $f\in L^2(\Omega)$. Since quasiconformal mappings possess the Luzin $N$-property, then $|J(x,\varphi)|^{-1}=|J(y,\varphi^{-1})|$ for almost all $x\in \Omega$ and for almost all $y=\varphi(x)\in \Omega'$.
So, the following inequality is correct:
\begin{multline*}
\left(\int\limits_{\Omega} |f(x)|^2~dx\right)^{\frac{1}{2}}
=\left(\int\limits_{\Omega} |f(x)|^2|J(x,\varphi)|^{-1}|J(x,\varphi)|~dx\right)^{\frac{1}{2}} \\
\leq \|J_{\varphi} \mid L^{\infty}(\Omega)\|^{-\frac{1}{2}} \left(\int\limits_{\Omega} |f(x)|^2|J(x,\varphi)|~dx\right)^{\frac{1}{2}}.
\end{multline*}

In turn by Theorem~\ref{Th4.1} we have
\begin{multline*}
\left(\int\limits_{\Omega} |f(x)|^2~dx\right)^{\frac{1}{2}}
\leq \|J_{\varphi^{-1}} \mid L^{\infty}(\Omega')\|^{\frac{1}{2}} \left(\int\limits_{\Omega'} |f \circ \varphi^{-1}(y)|^2~dy\right)^{\frac{1}{2}} \\
\leq K^{\frac{1}{2}} A_{2,2}({\Omega'}) \big\|J_{\varphi^{-1}}\mid L^{\infty}({\Omega'})\big\|^{\frac{1}{2}}
\left(\int\limits_{\Omega} |\nabla f(x)|^2~dx\right)^{\frac{1}{2}}
\end{multline*}
for any $f\in W^{1,2}_0(\Omega)$.
\end{proof}

\section{Lower estimates}

We consider the Dirichlet eigenvalue problem in the weak formulation:
\[
\int\limits_{\Omega} \left\langle \nabla u(x), \nabla v(x) \right\rangle dx
= \lambda \int\limits_{\Omega} u(x) v(x)~dx, \quad v \in W_0^{1,2}(\Omega).
\]

Recall that the first eigenvalue of the Dirichlet Laplacian is defined by

\[
\lambda_1(\Omega)= \inf_{u \in W_0^{1,2}(\Omega) \setminus \{0\}} \frac{\int\limits_{\Omega} |\nabla u(x)|^2~dx}{\int\limits_{\Omega} |u(x)|^2~dx}\,.
\]
In other words, this is the sharp constant in the Sobolev inequality
\[
\left( \int\limits_{\Omega} |u(x)|^2~dx \right)^{\frac{1}{2}} \leq A_{2,2}(\Omega)
\left( \int\limits_{\Omega} |\nabla u(x)|^2~dx \right)^{\frac{1}{2}}, \quad u \in W_0^{1,2}(\Omega).
\]

\begin{theorem}\label{l-est}
Let $\Omega$ be a $K$-quasiconformal $\infty$-regular domain about $\Omega'$. Then
\[
\lambda_1(\Omega) \geq \frac{\lambda_1(\Omega')}{K ||J_{\varphi^{-1}}\,|\,L^{\infty}(\Omega')||},
\]
where $J_{\varphi^{-1}}$ is a Jacobian of the $K$-quasiconformal mapping $\varphi^{-1}:{\Omega'}\to\Omega$.
\end{theorem}

\begin{proof}
According to the Min-Max Principle and Theorem~\ref{Th4.3} we have
\[
\int \limits_{\Omega}|f(x)|^2dx \leq A^2_{2,2}(\Omega) \int \limits_{\Omega}|\nabla f(x)|^2dx,
\]
where
\[
A_{2,2}(\Omega) \leq K^{\frac{1}{2}} A_{2,2}({\Omega'}) \big\|J_{\varphi^{-1}}\mid L^{\infty}({\Omega'})\big\|^{\frac{1}{2}}.
\]
Given the equality $A_{2,2}(\widetilde{\Omega})=\lambda_1(\widetilde{\Omega})^{-\frac{1}{2}}$,
$\widetilde{\Omega}=\Omega, \Omega'$, we get
\[
\lambda_1(\Omega) \geq \frac{\lambda_1(\Omega')}{K ||J_{\varphi^{-1}}\,|\,L^{\infty}(\Omega')||}.
\]
\end{proof}

From this theorem and the domain monotonicity property for the Dirichlet eigenvalues we obtain growth monotonicity estimates of the principal Dirichlet-Laplacian eigenvalue in quasiconformal regular domains.

\begin{theorem}\label{main}
Let $\Omega$ be a $K$-quasiconformal $\infty$-regular domain about $\Omega'$. Suppose $\Omega' \supset \Omega$, then
\[
\lambda_1(\Omega)-\lambda_1(\Omega') \geq \frac{1-K ||J_{\varphi^{-1}}\,|\,L^{\infty}(\Omega')||}{K ||J_{\varphi^{-1}}\,|\,L^{\infty}(\Omega')||}\lambda_1(\Omega'),
\]
where $J_{\varphi^{-1}}$ is a Jacobian of the $K$-quasiconformal mapping $\varphi^{-1}:{\Omega'}\to\Omega$.
\end{theorem}

\begin{proof}
By the theorem condition $\Omega' \supset \Omega$. Hence $\lambda_1(\Omega)\geq\lambda_1(\Omega')$.
By Theorem \ref{l-est} we have
\[
\lambda_1(\Omega) \geq \frac{\lambda_1(\Omega')}{K ||J_{\varphi^{-1}}\,|\,L^{\infty}(\Omega')||}.
\]
Performing simple actions in this inequality yields the required result, i.e.:
\[
\lambda_1(\Omega)-\lambda_1(\Omega') \geq \frac{1-K ||J_{\varphi^{-1}}\,|\,L^{\infty}(\Omega')||}{K ||J_{\varphi^{-1}}\,|\,L^{\infty}(\Omega')||}\lambda_1(\Omega').
\]
\end{proof}

In the case of quasiconformal mappings $\varphi:\Omega \to \mathbb D$, Theorem~\ref{main} can be reformulated as
\begin{theorem}\label{disk}
Let $\Omega$ be a $K$-quasiconformal $\infty$-regular domain about the unit disc $\mathbb D$. Suppose $\mathbb D \supset \Omega$, then
\[
\lambda_1(\Omega)-\lambda_1(\mathbb D) \geq \frac{1-K ||J_{\varphi^{-1}}\,|\,L^{\infty}(\mathbb D)||}{K ||J_{\varphi^{-1}}\,|\,L^{\infty}(\mathbb D)||} j_{0,1}^2,
\]
where $j_{0,1} \approx 2.4048$ is the first positive zero of the Bessel function $J_0$ and $J_{\varphi^{-1}}$ is a Jacobian of the $K$-quasiconformal mapping $\varphi^{-1}:\mathbb D \to \Omega$.
\end{theorem}

As examples, we consider the domains bounded by an epicycloid. Since
the domains bounded by an epicycloid are $K$-quasiconformal $\infty$-regular, we
can apply Theorem \ref{disk}, i.e.:

{\bf Example.} For $n \in \mathbb N$, the homeomorphism
\[
\psi(z)=A\left(z+\frac{z^n}{n}\right)+B\left(\overline{z}+\frac{\overline{z}^n}{n}\right), \quad z=x+iy, \quad A>B\geq0,
\]
is quasiconformal with $K=(A+B)/(A-B)$ and maps the unit disc $\mathbb D$ onto the domain $\Omega_n$ bounded by an epicycloid of
$(n - 1)$ cusps, inscribed in the ellipse with semi-axes $(A+B)(n+1)/n$ and $(A-B)(n+1)/n$.

We calculate the Jacobian of mapping $\psi$ by the formula \cite{Ahl66}:
\[
J(z,\psi)=|\psi_z|^2-|\psi_{\overline{z}}|^2.
\]
Here
$$
\psi_z=\frac{1}{2}\left(\frac{\partial \psi}{\partial x}-i\frac{\partial \psi}{\partial y}\right) \quad \text{and} \quad
\psi_{\overline{z}}=\frac{1}{2}\left(\frac{\partial \psi}{\partial x}+i\frac{\partial \psi}{\partial y}\right).
$$
A direct calculation yields
\[
\psi_z=A(1+z^{n-1}), \quad \psi_{\overline{z}}=B(1+\overline{z}^{n-1}).
\]
Hence
\[
J(z,\psi)=(A^2-B^2)|1+z^{n-1}|^2
\]
and
\[
||J_{\psi}\,|\,L^{\infty}(\mathbb D)||
=\esssup\limits_{|z|<1}\left[(A^2-B^2)|1+z^{n-1}|^2\right]
\leq 4(A^2-B^2).
\]

Note that in the case $A+B<1/2$ the quality $K||J_{\psi}\,|\,L^{\infty}(\mathbb D)|| \leq 4(A+B)^2<1$ and $\mathbb D \supset \Omega_n$.

Thus, for $A+B<1/2$, by Theorem~\ref{disk} we have
\[
\lambda_1(\Omega_n)-\lambda_1(\mathbb D) \geq \frac{1-4(A+B)^2}{4(A+B)^2} j_{0,1}^2.
\]

\textbf{Acknowledgements.} The author thanks Vladimir Gol'dshtein and Alexander Ukhlov for useful discussions and valuable comments. This work was supported by the Ministry of Science and Higher Education of Russia (agreement No. 075-02-2021-1392) and RSF Grant No. 20-71-00037.

\vskip 0.3cm

Division for Mathematics and Computer Sciences, Tomsk Polytechnic University, 634050 Tomsk, Lenin Ave. 30, Russia; Regional Scientific and Educational Mathematical Center, Tomsk State University, 634050 Tomsk, Lenin Ave. 36, Russia
							
\emph{E-mail address:} \email{vpchelintsev@vtomske.ru}   \\

\end{document}